\documentclass[12pt,oneside,reqno]{amsart}
\usepackage[all]{xy}
\usepackage{amsfonts,amsmath,oldgerm,amssymb,amscd}
\UseComputerModernTips
\numberwithin{equation}{section}
\usepackage[breaklinks]{hyperref}
\allowdisplaybreaks

\usepackage{enumerate}

\usepackage{caption} 
\newtheorem{theorem}{Theorem}
\newtheorem{proposition}[theorem]{Proposition}

\newtheorem{lemma}[subsection]{{\bf Lemma}}

\newtheorem{remark}[subsection]{Remark}

\newcommand{\al}{\alpha}

\newcommand{\ga}{\gamma}

\newcommand{\Z}{\mbox{$\mathbb Z$}}
\newcommand{\Q}{\mbox{$\mathbb Q$}}

\begin{document}

\title[Cullen numbers in sums of terms of recurrence sequence]{Cullen numbers in sums of terms of recurrence sequence} 

\author[N.K. Meher]{N.K. Meher}
\address{Nabin Kumar Meher,  National Institute of Science Education and Research, Bhubaneswar, HBNI, P.O. Jatni, Khurda, Odisha -752050, India.}
\email{mehernabin@gmail.com}

\author[S. S. Rout]{S. S. Rout}
\address{Sudhansu Sekhar Rout, Institute of Mathematics and Applications\\ Andharua, Bhubaneswar, Odisha - 751029\\ India.}
\email{lbs.sudhansu@gmail.com; sudhansu@iomaorissa.ac.in}

\maketitle
\pagenumbering{arabic}
\pagestyle{headings}

\begin{abstract}
Let $(U_n)_{n\geq 0}$ be a fixed linear  recurrence sequence of integers with order at least two, and for any positive integer $\ell$, let $\ell \cdot 2^{\ell} + 1$ be a Cullen number. Recently in \cite{bmt}, generalized Cullen numbers in terms of linear recurrence sequence $(U_n)_{n\geq 0}$ under certain weak assumptions has been studied.  However, there is an error in their proof. In this paper, we generalize their work, as well as our result fixes their error. In particular, for a given polynomial $Q(x) \in \mathbb{Z}[x]$ we consider the Diophantine equation $U_{n_1} + \cdots +  U_{n_k} = \ell \cdot x^{\ell} + Q(x)$, and prove effective finiteness result.  Furthermore, we demonstrate our method by an example.
\end{abstract}

\section{Introduction}\label{sec2}

Let $r$ be a positive integer. The linear recurrence sequence $(U_{n})_{n \geq 0}$ of order $r$ is defined as 
\begin{equation}\label{eq4}
U_{n} = a_1U_{n-1} + \dots +a_rU_{n-r}
\end{equation}
where $a_1,\dots, a_r \in \Z$ with $a_r\neq 0$ and $U_0,\dots,U_{r-1}$ are integers not all zero. 

The characteristic polynomial of $U_n$ is defined by
\begin{equation}\label{eq5}
f(x):= x^r - a_1x^{r-1}-\dots-a_r = \prod_{i =1}^{t}(x - \alpha_i)^{m_i}\in \Z[X]
\end{equation}
where $\alpha_1,\dots,\alpha_t$ are distinct algebraic numbers and $m_1,\dots, m_t$ are positive integers. Then  $U_{n}$ (see e.g. Theorem C1 in part C of \cite{st}) has a nice  representation of the form
\begin{equation}\label{eq6}
U_n=\sum_{i=1}^t f_i(n)\alpha_{i}^n \ \ \ \text{for all}\ n\geq 0,
\end{equation}
where $f_i(x)$ is a polynomial of degree $m_i -1$ $(i=1,\dots,t)$ and this representation is uniquely determined. We call the sequence $(U_n)_{n \geq 0}$ {\it simple} if $t=r$. In this paper, we assume that $t\geq 2,$  and the characteristic polynomial $f$ is irreducible over $Q$. Thus, all the roots $\alpha_i, (1\leq i \leq r)$ of \eqref{eq5} are simple roots  and hence $f_i(n)$ in \eqref{eq6} are constants, say $f_i$ (because the degree of $f_i(n)$ would be at most $m_i - 1 = 1 - 1 = 0)$ and hence \eqref{eq6} becomes
\begin{equation}\label{eq6a}
U_n=\sum_{i=1}^r f_i\alpha_{i}^n \ \ \ \text{for all}\ n\geq 0.
\end{equation}
 If $|\alpha_1|>|\alpha_j|$ for all $j$ with $2\leq j\leq r$, then we say that $\alpha_1$ is a dominant root of the sequence $(U_n)_{n\geq 0}$.

 The {\em Cullen numbers} are elements of the sequence $(C_{\ell})_{\ell\geq 0}$, where the $\ell$-th term of the sequence is given by $C_{\ell}:= \ell \cdot 2^{\ell} + 1, \hbox{with} \ \ell \in \mathbb{Z}_{\geq 0}$. This sequence was first introduced by Father J. Cullen \cite{cullen} and it is also mentioned in Guy's book \cite[Section B20]{guy}. In 1976, C. Hooley \cite{Hooley1976} proved that almost all Cullen numbers are composite. However, there is a conjecture that there are infinitely many {\it Cullen primes}. One of the Cullen prime having more than $2$ million digits is $C_{6679881}.$

 Further, we define the {\em generalized Cullen numbers}  are the numbers of the form
 \[C_{\ell, s} = \ell \cdot s^{\ell} + 1\]
 where $\ell \geq 1$ and $s\geq 2$. Clearly, $C_{\ell, 2} = C_{\ell}$ for all $\ell \geq 1$. For simplicity we call $C_{\ell, s}$ as $s$-Cullen numbers.

The occurrence of Cullen numbers in recurrence sequence has been analyzed by various authors. For instance, Luca and Shparlinski \cite{ls1} studied on the pseudoprime Cullen numbers and Berrizbeitia et. al., \cite{bfg} investigated  on Cullen numbers which are both Riesel and Sierpinski numbers. Further, Luca and St$\breve{a}$nic$\breve{a}$ \cite{ls} proved that there are only finitely many Cullen numbers in a binary recurrence sequence under some additional assumptions. Besides this, they showed that the largest Fibonacci number in the Cullen sequence is $F_4$.  Recently, Bilu et. al., \cite{bmt} studied the occurrence of generalized Cullen numbers in a fixed linear recurrence sequences. In particular, they proved that there are finitely many solutions in integers $(n, m, x)$ of the Diophantine equation 
 \begin{equation}\label{cullen1}
 G_n = m \cdot x^{m} + T(x)
 \end{equation}
 for a given polynomial $T(x) \in \Z[x]$ with some assumptions on the roots of characteristic polynomial of the given linear recurrence sequence $(G_n)$ of order at least two.
 
However, there is an error in their proof \cite[Theorem 1]{bmt}. For instance, consider the linear recurrence sequence of order three defined by the recurrence relation
\[G_{n+3} = 3G_{n+2} -3G_{n+1} + 2G_{n},\quad n\geq 0\]
with initial values $G_0 = 0, G_1 = 1$ and $G_2 = 1$. The characteristic polynomial of $G_n$ is $x^3 -3x^2 +3x-2$ and its roots are $2, (1\pm \sqrt{3}i)/2$. Here these roots satisfy all the technical conditions in \cite[Theorem 1]{bmt}. Taking $T(x) =-1$ in  \eqref{cullen1}, we get
\begin{equation}\label{cullen2}
G_n = m\cdot x^m -1. 
\end{equation}
The integer solutions of \eqref{cullen2} are precisely 
\[(n, m, x) \in \{(6k+1, 1, 2), (6k+2, 1, 2)\mid k\in \Z_{\geq 0}\}
\]
 which shows that the conclusion in \cite[Theorem 1]{bmt} is false.
 
 \subsection{Main results}
 In this paper, for a given polynomial $Q(x)\in \Z[x]$ we study the following Diophantine equation 
 \begin{equation}\label{eq1}
 U_{n_1} + \cdots +  U_{n_k} = \ell \cdot x^{\ell} + Q(x)
 \end{equation}
in non negative integers $ n_1, \ldots, n_k, \ell$  with $ n_1 > n_2 >  \cdots > n_k \geq 0$. Our main result is the following.
\begin{theorem}\label{thm1}
Let $(U_n)_{n\geq 0}$ be a linear recurrence sequence of order at least two such that its characteristic polynomial is irreducible over $\Q$ and has a real dominant root $\alpha_1>1$ and let $Q(x) \in \Z[x]$ be a polynomial. Then there exists an effectively computable constant $C$ depending only on $(U_n)_{n\geq 0},$ and $Q(x)$ such that the solutions $(n_1, n_2, \dots, n_k, \ell )$ of equation \eqref{eq1} satisfy 
\[\max \{ n_1, n_2, \dots, n_k, \ell    \} <C (\log |x|)^k . \]
\end{theorem}
 \begin{remark} 
We note that if $k =1$ in \eqref{eq1}, then we do not need the assumptions $\alpha_1>1$ and $\alpha_1$ is real in Theorem \ref{thm1}. Also, we obtain the conclusion in \cite[Theorem 1]{bmt}.
\end{remark}

 Fibonacci sequence $(F_n)_n\geq {0}$ is a well known recurrence sequence of order two which satisfies the recurrence relation
\[F_{n} = F_{n-1}+ F_{n-2}, \quad n\geq 2\]
with initial values $F_0 =1$ and $F_1=1$. Thus, from \eqref{eq6a} we have the Binet form
\begin{equation}\label{fibeq02}
F_n = \frac{\alpha^n- \beta^n}{\sqrt{5}}
\end{equation}
where $\alpha = (1+\sqrt{5})/2$ and $\alpha \beta= -1$.  From the well known Binet form, we deduce the bound of $F_n$ as
\begin{equation}\label{fibeq03}
\alpha^{n-2}\leq F_n \leq \alpha^{n-1}.
\end{equation}

Our next theorem illustrates Theorem \ref{thm1} for Fibonacci sequence. 

\begin{theorem} \label{thm2}
If $(n_1, n_2 , \ell )$ is a solutions of the Diophantine equation 
\begin{equation}\label{fibeq01}
F_{n_1}+ F_{n_2}= \ell\cdot  2^{\ell} + 1
\end{equation}
in  non-negative integers $  n_1, n_2, \ell$ with $ n_1 \geq n_2\geq 0,$ then 
\[ \max\{  n_1, n_2, \ell \} < 7 \times 10^{18}. \]
\end{theorem}

\begin{remark}
However, we expect that if \eqref{fibeq01} holds with $n_1\geq n_2 \geq 0$, then
\[(n_1, n_2, \ell) \in \{(1, 0, 0), (2, 0, 0), (4, 0, 1), (3, 1, 1), (3, 2, 1), (6, 1, 2), (6, 2, 2), (14, 6, 6)\}.\]
\end{remark}

\section{Auxiliary results}

Let $\eta$ be an algebraic number of degree $d$ with minimal polynomial 
\[c_{0}x^d + c_1x^{d-1} + \cdots + c_d = c_0 \prod_{i=1}^{d}\left(X - \eta^{(i)}\right),\]
where $c_0$ is the leading coefficient of the minimal polynomial of $\eta$ over $\mathbb{Z}$ and the
$\eta^{(i)}$'s are conjugates of $\eta$ in $\mathbb{C}$.  We define the absolute {\it logarithmic height} of an algebraic number $\eta$ as 
\[
h(\eta) = \frac{1}{d} \left( \log |c_0| + \sum_{i=1}^d \log \max ( 1, |\eta^{(i)}| ) \right). 
\]
In particular, if $\eta = p/q$ is a rational number with $\gcd(p, q) = 1$ and $q >0$, then $h(\eta) = \log \max\{|p|, q\}$.

To prove our theorem, we use lower bounds for linear forms in logarithms to get a bound for $\max\{n_1, \ldots, n_k, \ell\}$ appearing in \eqref{eq1}. Generically, we need the following general lower bound for linear forms in logarithms due to Matveev \cite[Theorem 2.2]{Matveev2000}.
\begin{lemma}[\cite{Matveev2000}]\label{lem12}
Let $\ga_1,\ldots,\ga_s$ be real algebraic numbers and let $b_{1},\ldots, b_{s}$ be non-zero rational integer numbers. Let $D$ be the degree of the number field $\mathbb{Q}(\ga_1,\ldots,\ga_s)$ over $\mathbb{Q}$ and let $A_{j}$ be real numbers satisfying 
\begin{equation}\label{eq8a}
A_j \geq \max \left\{ Dh(\ga_j) , |\log \ga_j|, 0.16  \right\}, \quad j= 1, \ldots,s.
\end{equation}
 
Assume that $B\geq \max\{|b_1|, \ldots, |b_{s}|\}$ and $\Lambda:=\ga_{1}^{b_1}\cdots\ga_{s}^{b_s} - 1$. If $\Lambda \neq 0$, then
\[|\Lambda| \geq \exp \left( -1.4\times 30^{s+3}\times s^{4.5}\times D^{2}(1 + \log D)(1 + \log B)A_{1}\cdots A_{s}\right).\]
\end{lemma}

\begin{lemma}\label{lem13a}
Suppose that $(U_n)_{n\geq 0}$ has a real simple dominant root $\alpha_1$  with $1<\alpha_1\not \in \mathbb{Z}$ and $f_1$ is the constant coefficient of $\alpha_1^n$ defined in the formula \eqref{eq6}. Set
\[\Lambda_i = 1 - \ell f_{1}^{-1}x^{\ell}\alpha_1^{-n_1}\left(1 + \al_1^{n_{2} - n_1} + \cdots + \al_1^{n_{i} - n_1} \right)^{-1}\] 
for all $2\leq i \leq k$. If  $\Lambda_i = 0$, then there exists an index $m$ with $2\leq m \leq t$ such that $n_k < \varkappa_i$, where
\begin{equation}\label{eq26yz}
\varkappa_i:=  \begin{cases}
\frac{\log (i\cdot (|f_1^{(m)}|/|f_1|))}{\log \alpha_1} &\quad \text{if} \;\;|\alpha_m| \leq 1\\
\frac{\log (i\cdot (|f_1^{(m)}|/|f_1|))}{\log (\alpha_1/|\alpha_m|)} &\quad \text{if} \;\; |\alpha_m| > 1.
\end{cases}
\end{equation}
\end{lemma}

\begin{proof}
 Suppose that $\Lambda_i = 0$. This implies
\begin{equation}\label{eq25}
\ell x^{\ell} = f_{1}\cdot \left(\alpha_1^{n_1} + \cdots + \al_1^{n_{i}}\right). 
\end{equation}
Since $\alpha_1\notin \mathbb{Z}$, then there exists a conjugate $\alpha_m$ of $\alpha_1$ in the field $\mathbb{Q}(\alpha_1, \ldots, \alpha_t)$ such that $\alpha_1\neq \alpha_m$. Therefore, by taking the $m$-th conjugate of both sides of \eqref{eq25}, we get
\begin{equation}\label{eq26}
\ell x^{\ell}  = f_1^{(m)}\cdot \left(\alpha_m^{n_1}+ \cdots + \al_m^{n_{i}}\right),
\end{equation}
 (here $f_1^{(m)}$ is $m$-th conjugate of $f_1$).  As $\al_1>1$ and it is real, we deduce from \eqref{eq25} and \eqref{eq26},
\begin{align*}
|f_1|\alpha_1^{n_1} & < |f_1||\alpha_1^{n_1} + \cdots + \al_1^{n_{i}}| = | f_1^{(m)}||\alpha_m^{n_1}+ \cdots +\al_m^{n_{i}}|.
\end{align*}
This implies
\begin{equation}\label{eq26z}
\alpha_1^{n_1} < i \cdot (|f_1^{(m)}|/|f_1|) |\alpha_m|^{n_1}
\end{equation}
 and hence $n_1 < \varkappa_i$ where $\varkappa_i$ is given in \eqref{eq26yz}.
 \end{proof}
 
 \begin{remark}\label{lem13b}
 Suppose $\Lambda_1 := 1 - \ell f_{1}^{-1}x^{\ell}\alpha_1^{-n_1}=0$. Then $\ell x^{\ell} = f_1 \alpha_1^{n_1}$. So by taking conjugation of this relation in $\Q(\alpha_1, \ldots, \alpha_t)$, we get $\ell x^{\ell} = f_1^{(m)} \alpha_m^{n_1}$, where  $f_1^{(m)}$ is the conjugate of $f_1$ over $\Q(\alpha_1, \ldots, \alpha_t)$. Thus by taking absolute values, we obtain $|\alpha_1/\alpha_m|^{n_1} = |f_1^{(m)}/f_1|$ and this implies $n_1 = \log (|f_1^{(m)}/f_1|)/\log(|\alpha_1/\alpha_m|)=: \varkappa_1$.
 \end{remark}
 
\begin{proposition}\label{prop1}
Let $(U_n)_{n\geq 0}$ be a linear recurrence sequence of order at least two such that its characteristic polynomial is irreducible and has a real dominant root $\alpha_1>1$.  If equation \eqref{eq1} holds with $n_{1} > n_{1} > \cdots > n_{k}$, 
 then for $ 2 \leq i \leq k$ we have, 
\begin{equation}\label{boundeq}
(n_{1}-n_{i})  \leq C_i  (\log x)^{i-1}(\log n_1)^{2i- 2}.
\end{equation}
\end{proposition}
\begin{proof}
Without loss of generality, we may assume that $x\geq 2$ in \eqref{eq1} and $|\alpha_2|\geq \cdots \geq |\alpha_r|$. Suppose $\varkappa:= \max\{\varkappa_1, \cdots, \varkappa_{k}\}$, where $\varkappa_i \;(1\leq i\leq k)$ are defined in Lemma \ref{lem13a}) and Remark \ref{lem13b}. Now, we  may assume $n_1 > \varkappa$.   In the proof, $c_1, \ldots, c_{23}$ denote positive effective constants depending on $(U_n)_{n\geq 0}$ and $Q(x)$.

We use induction method  to find an upper bound of $n_1- n_i$. We first calculate the upper bound of $n_1 - n_2.$ Let us consider the equation \eqref{eq1} and rewrite this as
\begin{equation}\label{eq7}
\ell \cdot x^{\ell} - f_1 \alpha_1^{n_1} =  \sum_{i=2}^{k} U_{n_i} + \sum_{i =2}^{r}f_i\alpha_i^{n_1}  - Q(x).
\end{equation}
From, \eqref{eq6a}, we obtain (see \cite[Lemma 3.2]{bhpr})
 \begin{equation}\label{eq77}
    |U_n|\leq c_1|\alpha_1|^n\ \ \ (n\geq 1).
  \end{equation}
Hence from \eqref{eq7} and \eqref{eq77},
\begin{equation}\label{eq8}
|\ell x^{\ell} - f_1 \alpha_1^{n}| \leq c_2 |\alpha_1|^{n_2} + c_3|\alpha_2|^{n_1} + |Q(x)| \leq c_4 (|\alpha_1|^{n_2} + |\alpha_2|^{n_1}).
\end{equation}
Dividing both sides of the  equation \eqref{eq8} by $|f_1 \alpha_1^{n_1}|$, we  get
\begin{equation}\label{eq9}
|1 - \ell f_1^{-1} \alpha_1^{-n_1}x^{\ell}| \leq c_5  |\alpha_1|^{n_2 - n_1}.
\end{equation}
We denote $\Lambda_1:= 1 - (\ell f_1^{-1}) \alpha_1^{-n_1}x^{\ell}$ and by Remark \ref{lem13b}, $\Lambda_1 \neq 0$ as $n_1 > \varkappa$. In order to apply Lemma \ref{lem12}, we consider 
\[D : = \Q(\ga_1, \ldots, \ga_s)/\Q, s:= 3, \;\; \ga_1 := \ell f_1^{-1}, \;\; \ga_2:= \alpha_1, \;\; \ga_3:= x\]
and 
\[b_1:= 1, \;\; b_2:= -n_1, \;\; b_3:= \ell.\]
Note that
\[|\ell x^{\ell} + Q(x)| \geq |\ell x^{\ell}| - |\Q(x)|.\]
 Thus, we have
\begin{align}\label{eq9a1}
\begin{split}
|\ell x^{\ell}| &\leq |\ell x^{\ell} + Q(x)| + |Q(x)| = |U_{n_1}+ \cdots  + U_{n_k}| +  |Q(x)|  \\
& \leq k|U_{n_1}| + |Q(x)|  \leq c_6 \alpha_1^{n_1}.
\end{split}
 \end{align}
and this implies $ \ell \log x \leq c_7  n_1 \log \alpha_1$. Hence,
\begin{equation}\label{eq9a}
\ell \leq c_7  n_1 \log \alpha_1/\log x< c_{8}n_1,
\end{equation}
 with $c_8 :=\max\{c_7\log \alpha_1, c_7\log \alpha_1/\log 2\}$. Now choose $B = \max\{|b_1|, |b_2|, |b_3|\}= c_{9}n_1 $,  
\[h(\ga_1) \leq h(\ell)+ h(f_1) \leq c_{10} \log \ell, \; h(\ga_2) \leq \log \alpha_1,  \hbox{and} \;\; h(\ga_3) \leq \log x.\]
By applying Lemma \ref{lem12} and  the inequality \eqref{eq9a}, we obtain
\begin{align}\label{eq9a2}
 \log | \Lambda_1 |\geq -c_{11} (1+\log {n_1}) \log x \log \ell\geq -c_{12} \log x (\log n_1)^2.
\end{align}
 On comparing the lower and upper bound of $\Lambda_1$, we get
 \[(n_1-n_2) \leq c_{13}\log x (\log n_1)^2 \]
Therefore, the inequality \eqref{boundeq}  is true for $i=2$. We will assume that the inequality  \eqref{boundeq} is true for $i=q$ with $ 2 \leq q \leq k-1$.  Now, we prove that  the inequality  \eqref{boundeq} is true for $i =k$.
  Rewrite the equation \eqref{eq1} as 
  \begin{equation}\label{eq010}
\ell x^{\ell} - f_1 \alpha_1^{n_1}  \left( 1 +  \alpha_1^{n_{2} - n_1}  \cdots + \alpha_1^{n_{k-1} -n_1} \right) =  U_{n_k} +  \sum_{j=1}^{k-1} \sum_{i =2}^{r}f_i\alpha_i^{n_j}  - Q(x).
\end{equation}
Simplifying \eqref{eq010} with \eqref{eq77}, we obtain 
\begin{align}\label{eq10a}
\left| \ell x^{\ell} - f_1 \alpha_1^{n_1}  \left( 1 +  \alpha_1^{n_{2} - n_1}  \cdots + \alpha_1^{n_{k-1} -n_1} \right) \right|   & \leq  c_{14} \alpha_1^{n_k} + c_{15} |\alpha_2|^{n_1} + |Q(x)| \\ \notag
&  \leq c_{16} (\alpha_1^{n_k} + |\alpha_2|^{n_1}).
\end{align}
  Dividing both side of the equation \eqref{eq10a} by $  \left| f_1 \alpha_1^{n_1}  \left( 1 +  \alpha_1^{n_{2} - n_1}  \cdots + \alpha_1^{n_{k-1} -n_1} \right) \right| ,$
  we get 
 
\begin{align}\label{eq10b}
\left| 1 -  \ell f_1^{-1}   x^{\ell} \alpha_1^{- n_1}  \left( 1 +  \alpha_1^{n_{2} - n_1}  \cdots + \alpha_1^{n_{k-1} -n_1} \right)^{-1} \right|   
&  \leq c_{17}\alpha_1^{n_k - n_1}. 
\end{align} 
    Here we denote $\Lambda_{k-1} = 1 -  \ell f_1^{-1}   x^{\ell} \alpha_1^{- n_1}  \left( 1 +  \alpha_1^{n_{2} - n_1}  \cdots + \alpha_1^{n_{k-1} -n_1} \right)^{-1}.$ We will use Lemma \ref{lem12} to compute the lower bound of $\Lambda_{k-1}$. As  $n_1 > \varkappa$, we infer that $\Lambda_{k-1} \neq 0$. To use Lemma \ref{lem12}, we choose
    $$ D : = \Q(\ga_1, \ldots, \ga_s)/\Q, \;\; s:= 4,$$ $$  \;\; \ga_1 := \ell f_1^{-1}, \;\; \ga_2:= \alpha_1, \;\; \ga_3:= x, \gamma_4= \left( 1 +  \alpha_1^{n_{2} - n_1}  \cdots + \alpha_1^{n_{k-1} -n_1} \right) $$
and 
\[b_1:= 1, \;\; b_2:= -n_1, \;\; b_3:= \ell,\;\; b_4 = -1.\]
 Now choose $B = \max\{|b_1|, |b_2|, |b_3|, |b_4|\}= c_{18}n_1 .$ We have already computed $h(\gamma_1),$ $h(\gamma_2)$  and $h(\gamma_3).$ Now will estimate $h(\gamma_4).$
  By induction hypothesis,  we obtain
  $$ h(\gamma_4) \leq c_{19} (n_1- n_{k-1}) \log \alpha_1 \leq c_{20}  (\log x)^{k-2}  (\log n_1)^{2k-4} \log \alpha_1.$$  
  By applying Lemma \ref{lem12}, we have
  \begin{align}
  \begin{split}
\log  \Lambda_{k-1}& \geq - c_{21} (1+ \log n_1) \log \ell \log x \log \alpha_1   (\log x)^{k-2} (\log n_1)^{2k-4}\\
 & \geq -  c_{22}  (\log x)^{k-1}(\log n_1)^{2k-2}. 
 \end{split}
  \end{align}
   By comparing both upper and lower bound of $\Lambda_{k-1}$, we get
   $$ (n_1- n_k) \leq c_{23} (\log x)^{k-1} ( \log n_1)^{2k-2}$$ and this completes the proof of the proposition.
 \end{proof}

\section{Proof of Theorem \ref{thm1}}
First we assume that $n_1 > \varkappa$, where $\varkappa:= \max\{\varkappa_1, \cdots, \varkappa_{k-1}\}$, are defined in Lemma \ref{lem13a}). Without loss of generality, we assume $|\alpha_2|\geq \cdots \geq |\alpha_r|$. In the proof, $c_{24}, \ldots, c_{31}$ denote positive effective constants depending on $(U_n)_{n\geq 0}$ and $Q(x)$. 

Now rewrite \eqref{eq1} as 
\begin{equation}\label{eq12}
\ell x^{\ell} - f_1(\alpha_1^{n_1} +\alpha_1^{n_2}+ \cdots +  \alpha_1^{n_k})=\sum_{j =1}^{k}  \sum_{i =2}^{r}f_i\alpha_i^{n_j}  - Q(x).
\end{equation}
Taking absolute values on both sides of \eqref{eq12}
\begin{equation*}
|\ell x^{\ell} - f_1(\alpha_1^{n_1} +\alpha_1^{n_2}+ \cdots +  \alpha_1^{n_k})| \leq c_{24}|\alpha_2|^{n_1} + |Q(x)| \leq c_{25}  |\alpha_2|^{n_1}.
\end{equation*}
Since $\alpha_1$ is a dominant root, there exists $\delta \in (0, 1)$ such that
\begin{equation}\label{eq13}
|\ell x^{\ell} - f_1(\alpha_1^{n_1} +\alpha_1^{n_2}+ \cdots +  \alpha_1^{n_k})| \leq c_{26} |\alpha_1|^{(1-\delta)n_1}.
\end{equation}
Dividing both sides of the  equation \eqref{eq13} by $|f_1(\alpha_1^{n_1} +\alpha_1^{n_2}+ \cdots +  \alpha_1^{n_k})|$, we get
\begin{align}\label{eq14}
\begin{split}
|1 - \ell f_1^{-1}  x^{\ell} \alpha_1^{-n}&( 1 + \alpha_1^{n_2 - n_1}+ \cdots + \alpha_1^{n_k - n_1})^{-1}| \\
&\leq \frac{ c_{27}\alpha_1^{(1-\delta)n_1}}{|\alpha_1|^{n_1}| ( 1 + \alpha_1^{n_2 - n_1}+ \cdots + \alpha_1^{n_k - n_1}) |}\leq c_{28}\alpha_1^{-\delta n_1}.
\end{split}
\end{align}
Thus, our required linear form is 
\[\Lambda_k:= 1 - \ell f_1^{-1}  x^{\ell} \alpha_1^{-n}( 1 + \alpha_1^{n_2 - n_1}+ \cdots + \alpha_1^{n_k - n_1})^{-1}\]
and $\Lambda_k \neq 0$ since $n_1>\varkappa$ (see Lemma \ref{lem13a}). 

By applying Lemma \ref{lem12} using Proposition \ref{prop1}, we obtain
\begin{equation}\label{eq15}
\log | \Lambda_k |\geq -c_{29} (1+\log n_1) \log \ell \log x \log \alpha_1 (\log x)^{k-1} (\log n_1)^{2k-2} \geq -c_{30} (\log x)^{k}(\log n_1)^{2k}.
\end{equation}
Now from \eqref{eq14} and \eqref{eq15}, we have
\[n_1 \leq c_{31} (\log x)^{k}(\log n_1)^{2k}.\]
Therefore, $n_1< C (\log x)^{k} ,$ where $C$ depend upon $(U_n)_{n \geq 0}$ and $Q(x).$ This completes the proof of Theorem \ref{thm1}.

\section{Proof of Theorem \ref{thm2}}

For the reason of symmetry in \eqref{fibeq01}, we assume that $n_1 \geq n_2$. Firstly, consider that $n _1 = n_2$. Then, \eqref{fibeq01} becomes 
\begin{equation}\label{331}
2F_{n_1} = \ell \cdot 2^{\ell}+1.
\end{equation}
Since the left hand side of \eqref{331} is even and the right hand side is odd, \eqref{331} has no solution.  Now, assume that  $n _1 > n_2$. Further, if $n_2 = 0$, then \eqref{fibeq01} becomes 
\[F_{n_1} = \ell \cdot 2^{\ell}+1,\]
and the solutions are $(n_1, \ell) \in \{ (1, 0), (2, 0), (4, 1)$ (see \cite{ls}). From now on, assume that $n _1 > n_2 > 0$. From \eqref{eq9a1}, we get 
\[|\ell \cdot 2^{\ell}| \leq 2 |F_{n_1}| + 1.\]

Now, we will work on the assumption that $n_1 > n_2.$ 
 
\subsection{Bounding $(n_1 - n_2)$ in terms of $n_1$} 
From Eq.\eqref{fibeq01} and \eqref{fibeq03}, we get
\begin{equation}\label{eq33}
 2^{\ell}< \ell \cdot 2^{\ell} +1 = |F_{n_1} + F_{n_2}|\leq 2 \al^{n_1 - 1}.
\end{equation}
Taking logarithms on both sides of the inequality \eqref{eq33}, we obtain
\begin{equation}\label{eq34}
\ell \leq 1+ (n_1 - 1)\frac{\log \al}{\log 2} \leq 0.75 n_1.
\end{equation}
 Using $ |\beta| < 1 $ and $ \ell \leq 0.75  n_1$, similar to the inequality \eqref{eq9}, we obtain
\begin{align}\label{eq35}
\begin{split}
\left|1 -   \ell 2^{\ell}\alpha^{-n_1} \sqrt{5} \right|  \leq (2\sqrt{5}+1)\al^{n_{2} -n_1}.  
\end{split}
\end{align}

Suppose $1 -   \ell 2^{\ell}\alpha^{-n_1} \sqrt{5}  = 0$, then
\begin{equation}\label{eq42a}
 \ell 2^{\ell} \sqrt{5} = \al^{n_1}.
\end{equation} 
Taking squares on the both sides of \eqref{eq42a}, we arrive at a contradiction as  the left hand side of resulting equation is rational whereas the right hand side is irrational. In order to apply Lemma \ref{lem12}, we take 
$B = n_1, h(\gamma_1) = \log 2 = 0.6931\cdots < 0.7 , h(\gamma_2)=(\log \al)/2 = 0.2406\cdots < 0.25, h(\gamma_3) = \log(\sqrt{5}) < 0.81,$ $ h(\gamma_4) = \log \ell <  \log n$. We can choose $A_1 = 1.5, A_2 =  0.5, A_3 =  1.7,  A_4 =   2 \log n$. Using these parameters, we obtain 
\[\exp( - 1.3 \times 10^{14}\times (1 + \log n_1) \log n_1) < (2\sqrt{5}+1)\al^{n_{2} -n_1}.\] 
Further, since $(1+\log n_1) < 2\log n_1$ and $\log \ell < \log n_1$, we get 
\begin{equation}\label{eq36}
(n_1-n_2)\log \alpha <  3.5\times 10^{14}   (\log n_1)^2,
\end{equation}
which leads to
\begin{equation}\label{eq37}
(n_1-n_2) < \frac{3.5\times 10^{14}}{\log \alpha}  (\log n_1)^2 < 7.27 \times 10^{14} (\log n_1)^2.
\end{equation}
Similarly, the inequality corresponding to  second linear form is 
\begin{equation}\label{eq38}
| 1-  \ell 2^{\ell} \sqrt{5} \alpha^{-n_1}(1 + \alpha^{n_2-n_1})^{-1}|\leq (2\sqrt{5}+1)\alpha^{-n_1}. 
\end{equation}
Here, we choose $ h(\gamma_1) = \log 2 = 0.6931\cdots < 0.7 , h(\gamma_2)=(\log \al)/2 = 0.2406\cdots < 0.25, \gamma_3 = \sqrt{5} (1 + \alpha^{n_2-n_1})^{-1}| < 0.81,$ $ h(\gamma_4) = \log \ell <  \log n.$ In this case, $A_1$ and $A_2$ are same as in the previous case, and since
\[2(\log \sqrt{5} + (n_1 - n_2)\frac{\log \al}{2} + \log 2) <  2.1 \times 10^{15} (\log n_1)^2, \]
we take $A_3 := 2.1 \times 10^{15}  (\log n_1)^2$. Again, employing Lemma \ref{lem12}, we obtain
\begin{equation}\label{eq39}
\exp(- 7.3 \times 10^{11}\times (1 + \log n_1)( 2.1 \times 10^{15} (\log n_1)^2) \log n_1 )\leq (2\sqrt{5}+1)\alpha^{-n_1}
\end{equation}
and this implies
\begin{equation}\label{eq40}
n_1  < \frac{3.3\times 10^{27}}{\log \alpha} (\log n_1)^4 < 6.9 \times 10^{27}(\log n_1)^4.  \end{equation}
Further, using reduction procedure based on the LLL-algorithm \cite{s99}, we obtain $n_1- n_2 < 230$ and hence $n_1< 7\times 10^{18}$. This completes the proof.

\begin{remark}
To explicitly find all the solutions of \eqref{fibeq01}, one needs to further reduce the size of $n_1$ and the usual method for this process is Baker-Davenport reduction method (or results related to Dujella-Peth\"o theorem). However, for this problem, we have a form like
\[\ell \left(\frac{\log 2}{\log \alpha}\right) - n_1 +  \left(\frac{\log \ell \sqrt{5}}{\log \alpha}\right).\]
In this case, to use the reduction method, we should get a positive lower bound for $\epsilon$ depending on $\ell$ which by its size $\approx 10^{15}$ makes the calculation impossible.
\end{remark}

\end{document}